\documentclass[a4paper, draft, 12pt]{amsproc}
\usepackage{amssymb}
\usepackage{amsmath,amssymb, ulem, color}

\usepackage[hyphens]{url} 
\usepackage[dvips]{graphicx} 
\usepackage{cmdtrack} 
\usepackage{mathrsfs}                       

\theoremstyle{plain}
 \newtheorem{theorem}{Theorem}[section]
 \newtheorem{prop}{Proposition}[section]
 \newtheorem{lem}{Lemma}[section]
 
\theoremstyle{Definition}
 \newtheorem{exm}{Example}[section]
 
 \newtheorem{dfn}{Definition}[section]
\theoremstyle{remark}

 \numberwithin{equation}{section}

\renewcommand{\leq}{\leqslant}
\renewcommand{\geq}{\geqslant}
\renewcommand{\setminus}{\smallsetminus}
\title[Completely Invariant Escaping Set of Transcendental  Semigroup]{Completely Invariant Escaping Set of Transcendental Semigroup}

\subjclass[2010]{37F10, 30D05}

\keywords{Transcendental semigroup, escaping set, S-completely invariant escaping set, completely invariant escaping set etc.}

\author[B. H. Subedi]{\bfseries  Bishnu Hari Subedi}

\address{ 
Central Department of Mathematics \\ 
Institute of Science and Technology   \\ 
Tribhuvan University   \\ 
Kirtipur, Kathmandu\\
Nepal}
\email{subedi.abs@gmail.com / subedi\_bh@cdmathtu.edu.np }

\author[A. Singh]{Ajaya Singh}
\address{Central Department of Mathematics, Institute of Science and Technology, Tribhuvan University, Kirtipur, Kathmandu, Nepal }
\email{singh.ajaya1@gmail.com / singh\_a@cdmathtu.edu.np} 
\thanks{This research work of the first author is supported by PhD Faculty fellowship from University Grants Commission of Nepal} 

\begin{document}

{\begin{flushleft}\baselineskip9pt\scriptsize
\end{flushleft}}
\vspace{18mm} \setcounter{page}{1} \thispagestyle{empty}

\begin{abstract}
For a non-trivial transcendental semigroup, escaping set $ I(S) $ is in general S-forward invariant and it is S-completely invariant if semigroup $ S $ is abelian.  In the contrary of this result, we investigate completely invariant escaping set $ K(S) $ in different way even if semigroup $ S $ is not abelian and we discuss some properties and structure of such type of escaping set. Also, we establish some relations between completely invariant escaping set $ K(S) $ and the general escaping set $ I(S) $.
\end{abstract}

\maketitle

\section{Introduction} 

Throughout this paper, we denote the \textit{complex plane} by $\mathbb{C}$ and set of integers greater than zero by $\mathbb{N}$. 
We assume the function $f:\mathbb{C}\rightarrow\mathbb{C}$ is \textit{transcendental entire function}  unless otherwise stated. 
For any $n\in\mathbb{N}, \;\; f^{n}$ always denotes the nth \textit{iterates} of $f$. Let $ f $ be an entire function. The set of the form
$$
I(f) = \{z\in \mathbb{C}:f^n(z)\rightarrow \infty \textrm{ as } n\rightarrow \infty \}
$$
is called an \textit{escaping set} and any point $ z \in I(S) $ is called \textit{escaping point}. For a transcendental entire function $f$, the escaping set $I(f)$ was first studied by A. Eremenko \cite{ere}. He himself showed that 
 $I(f)\not= \emptyset$; the boundary of this set is a Julia set $ J(f) $ (that is, $ J(f) =\partial I(f) $);
 $I(f)\cap J(f)\not = \emptyset$; and 
 $\overline{I(f)}$ has no bounded component. By motivating from this last statement, he posed a question: \textit{Is every component of $ I(f) $ unbounded?}. This question is considered  as an important open problem of transcendental dynamics and nowadays it is famous as \textit{Eremenko's conjecture}. Note that the complement of Julia set $ J(f) $ in complex plane $ \mathbb{C} $ is a \textit{Fatou set} $F(f)$ and any maximally connected open subset of a Fatou set is called a \textit{Fatou component}. 
 
There are two types of points for which the inverse of entire map $ f $ is not well defined, namely critical values and asymptotic values.
Recall that the set $ C(f) = \{z\in \mathbb{C} : f^{\prime}(z) = 0 \}$ is the set of \textit{critical points} of the transcendental entire function $ f $ and  the set $CV(f) = \{w\in \mathbb{C}: w = f(z)\;\ \text{such that}\;\ f^{\prime}(z) = 0\} $ of all images of all critical points is called the set of \textit{critical values}.  The set 
$AV(f)$ consisting of all  $w\in \mathbb{C}$ such that there exists a curve (asymptotic path) $\Gamma:[0, \infty) \to \mathbb{C}$ so that $\Gamma(t)\to\infty$ and $f(\Gamma(t))\to w$ as $t\to\infty$ is called the set of \textit{asymptotic values} of $ f $ and the set
$SV(f) =  \overline{(CV(f)\cup AV(f))}$
is called the set of \textit{singular values} of $ f $.  Note that this set is coincide with the set of singularities of the inverse function $f^{-1}$ of $f$, and so this set is also denoted by $\text{Sing}(f^{-1})$.  Among the entire functions, only transcendental entire functions may have asymptotic values: clearly polynomials cannot have finite asymptotic values. If $SV(f)$ is a bounded set, then $f$ is said to be of \textit{bounded type}.                                                                                                                             The set $\mathscr{B} = \{f: f\;\  \textrm{is of bounded type}\}$
is called  \textit{Eremenko-Lyubich class}.

The main concern of this paper is to the study of the completely invariant escaping set under transcendental semigroup. So we start our formal study from the notion of transcendental semigroup. 
\begin{dfn}[\textbf{Transcendental semigroup and subsemigroup}]\label{ts}
Let $ A = \{f_i: i\in \mathbb{N}\} $ be a set of transcendental entire functions $ f_{i}: \mathbb{C}\rightarrow \mathbb{C} $. A \textit{transcendental semigroup} $S$ is a semigroup generated by the set $ A $ with semigroup operation being the functional composition. We denote this semigroup by $S = \langle A \rangle  = \langle f_{1}, f_{2}, f_{3}, \cdots, f_{n}, \cdots \rangle$. A non-empty subset $ T $ of transcendental semigroup $ S $ is a subsemigroup of $ S $ if $ f \circ g \in T $ for all $ f, \; g \in T $.
\end{dfn}

A semigroup generated by finitely many functions $f_i, (i = 1, 2, 3,\ldots, n) $  is called \textit{finitely generated transcendental semigroup}. We write $S= \langle f_1,f_2,\ldots,f_n\rangle$.
 If $S$ is generated by only one transcendental entire function $f$, then $S$ is \textit{cyclic transcendental  semigroup}. We write $S = \langle f\rangle$. In this case, each $g \in S$ can be written as $g = f^n$, where $f^n$ is the nth iterates of $f$ with itself. Note that in our forthcoming study of transcendental semigroup theory, we say $S = \langle f\rangle$ is a \textit{trivial semigroup}.  
  
 The transcendental semigroup $S$ is \textit{abelian} if  $f_i\circ f_j =f_j\circ f_i$  for all generators $f_{i}$ and $f_{j}$ of $ S $. We say semigroup $ S $ is  said to be of \textit{bounded type}  if each generator $ f_{i},\;( i \in \mathbb{N})$ of  $S$  is taken from the Eremenko-Lyubich class $ \mathscr{B} $. 
 
Based on the Fatou-Julia-Eremenko theory of a complex analytic function, the Fatou set, Julia set and escaping set in the settings of semigroup are defined as follows.
\begin{dfn}[\textbf{Fatou set, Julia set and escaping set}]\label{2ab} 
\textit{Fatou set} of the transcendental semigroup $S$ is defined by
  \[F (S) = \{z \in \mathbb{C}: S\;\ \textrm{is normal in a neighborhood of}\;\ z\}\] 
The \textit{Julia set} of $S$ is defined by $J(S) = \mathbb{C} - F(S)$
 and the \textit{escaping set} of $S$ by 
        \[I(S) = \{z \in \mathbb{C}: \;  f^n(z)\rightarrow \infty \;\ \textrm{as} \;\ n \rightarrow \infty\;\   \textrm{for all}\;\ f \in S\}\]
We call each point of the set $  I(S) $ by \textit{escaping point}.        
\end{dfn} 
Note that above definition of escaping set is slightly different from the definition  of escaping set by Dinesh Kumar and Sanjay Kumar {\cite[Definition 2.1]{kum2}}. In {\cite[Theorem 3.2]{sub1}}, we proved that our definition of escaping set is more general than the definition of Dinesh Kumar and Sanjay Kumar.  

The following immediate relation hold for any $ f \in S $  from the definition \ref{2ab} of escaping set. 

\begin{theorem}\label{1c}
Let $ S $  be a transcendental semigroup. Then
$I(S) \subset I(f)$ for all $f \in S$  and hence  $I(S)\subset \bigcap_{f\in S}I(f)$. 
\end{theorem}

From this theorem \ref{1c}, we can say that the escaping set may be empty. Note that $I(f)\not = \emptyset$  in classical iteration theory \cite{ere}. Dinesh Kumar and Sanjay Kumar {\cite [Theorem 2.5]{kum2}} have mentioned the following transcendental  semigroup $S$, where $I(S)$ is an empty set.
\begin{theorem}\label{e}
The transcendental semigroup $S = \langle f_{1}, f_{2}\rangle$  generated by two functions $f_{1}$ and $ f_{2} $ from  respectively two parameter families of functions 
$\mathscr{F} = \{e^{-z+\gamma}+c\;  \text{where}\;  \gamma, c  \in \mathbb{C} \; \text{and}\;  Re(\gamma)<0, \; Re(c)\geq 1\}$ 
and $\mathscr{F^{\prime}} =\{e^{z+\mu}+d, \; \text{where}\;  \mu, d\in \mathbb{C} \; \text{and}\; Re(\mu)<0,  \; Re(d)\leq -1\}$ has empty escaping set $I(S)$. \end{theorem}

There are also transcendental semigroups whose escaping sets are non-empty. 
If escaping set of the generators of transcendental semigroup $ S $  are equal, then it is not difficult to find escaping set of semigroup as shown in the following result. 
\begin{theorem}\label{1c1}
Let $  S =\langle f, g \rangle $ be a transcendental semigroup such that $ I(f) =I(g) $. Then $ I(S) = I(f) = I(g) $. 
\end{theorem} 
\begin{proof}
Since $ I(f) = I(g) $. It follows $ I(f^{n_{i}}) = I(g^{n_{j}}) $ for all $ n_{i}, n_{j}\in \mathbb{N} $.
 Since any $ h \in S $ can be written as $ h = f^{n_{k}} \circ g^{n_{{k-1}}} \circ \ldots \circ g^{n_{1}} $ where $ n_{k}, n_{k-1}, \ldots n_{1}\in \mathbb{N} $ and some $ n_{j} $ are allowed to be zero as well. These facts concluded that $ I(h) = I(f) =I(g) $ for all $ h \in S $. Therefore from the theorem \ref{1c},  $ I(S) = I(f) = I(g) $. 
\end{proof}
There are examples of transcendental entire functions that have same escaping sets. The following criteria is proved by Dinesh Kumar and Sanjay Kumar {\cite[Theorem 2.12]{kum3}}.
\begin{theorem}\label{1c2}
Let $ f $ be a transcendental entire functions of period p and let $ g = f^{k} + p, \; k\in \mathbb{N}$. Then $ I(f) =I(g) $. 
\end{theorem}
In {\cite[Lemma 3.3]{sub2}}, we have also proved the following criterion to have same escaping sets of two transcendental entire functions. 
\begin{theorem}\label{1c3}
Let $ f $ and $ g $ be two permutable transcendental entire functions of bounded type. Then $I(f) =I(g) $. 
\end{theorem}
If semigroup $ S $ is generated by two transcendental entire functions as mentioned in the theorems \ref{1c2} and \ref{1c3}, then by theorem \ref{1c1}, we have $ I(S) = I(f) = I(g) $.

 \section{Invariant features of escaping set}

We start our study from  invariant property  which is considered a very basic and fundamental structure of escaping set $ I(S) $ of transcendental semigroup $ S $. 
\begin{dfn}[\textbf{Forward, backward and completely invariant set}]
For a given semigroup $S$, a set $U\subset \mathbb{C}$ is said to be $ S $-\textit{forward invariant}\index{forward ! invariant set} if $f(U)\subset U$ for all $f\in S$. It is said to be $ S $-\textit{backward invariant}\index{backward ! invariant set} if $f^{-1}(U)\subset U$ for all $ f\in S $ and it is called $ S $-\textit{completely invariant}\index{completely invariant ! set} if it is both forward and backward invariant.
\end{dfn}
Already Dinesh Kumar and Sanjay Kumar {\cite[Theorem 4.1]{kum2}} and recently we {\cite[Theorem 2.3]{sub2}} prove the following result. Note that our proof is based on the definition \ref{2ab} of escaping set. 
\begin{theorem}\label{fi}
The escaping set $ I(S) $ of transcendental semigroup $ S $ is S-forward invariant.
\end{theorem}

It is well-known from \cite{ber1, ere,  hou, mor} that the classical  escaping set of transcendental entire function $f$ is completely invariant. 

There are several classes of transcendental semigroups from which we get  backward invariant escaping sets. In {\cite[Theorem 2.1]{kum1}} Dinesh Kumar and Sanjay Kumar and in {\cite[Theorem 2.6]{sub2}} we prove the following result. 

\begin{theorem}\label{1d}
The escaping set $ I(S) $ of transcendental semigroup $ S $ is S-backward invariant if  $ S $ is an abelian transcendental semigroup. 
\end{theorem}

With the result of this theorem \ref{1d}, we can conclude that  an escaping set $ I(S) $ is S-completely invariant if $ S $ is an abelian transcendental semigroup.  For example, the following semigroups 
\begin{itemize}
\item $\langle z +\gamma \sin z, \; z +\gamma \sin z + 2k \pi \rangle$,  
\item $\langle z +\gamma \sin z, \; -z -\gamma \sin z + 2k \pi \rangle$,
\item $\langle z + \gamma e^{z}, \; z + \gamma e^{z} + 2k \pi i \rangle$, 
\item  $\langle z - \sin z, \; z - \sin z + 2\pi \rangle$, 

\end{itemize} 
are abelian transcendental  semigroups, so their escaping sets are S-completely invariant. Note that it will not better to  conclude that escaping set can not be completely invariant unless the semigroup is abelian.

Eremenko's \cite{ere} result $\partial I(f) = J(f)$ of classical transcendental dynamics can be generalized to semigroup settings. Dinesh Kumar and Sanjay Kumar {\cite[Lemma 4.2 and Theorem 4.3]{kum2}} proved the following result.  
\begin{theorem}\label{3}
Let $S$ be a transcendental entire semigroup. Then
\begin{enumerate}
\item $int(I(S))\subset F(S)\;\ \text{and}\;\ ext(I(S))\subset F(S) $, where $int$ and $ext$ respectively denote the interior and exterior of $I(S)$.  
 \item $\partial I(S) = J(S)$, where $\partial I(S)$ denotes the boundary of $I(S)$. 
\end{enumerate}
\end{theorem}

Eremenko and Lyubich \cite{ere1} proved that if transcendental function $ f\in \mathscr{B} $, then $ I(f)\subset J(f) $, and $ J(f) = \overline{I(f)} $. Dinesh Kumar and Sanjay Kumar {\cite [Theorem 4.5]{kum2}} generalized these results to a finitely generated transcendental semigroup of bounded type  as shown below.
\begin{theorem}\label{4}
For every finitely generated transcendental semigroup $ S= \langle f_1, \\ f_2,  \ldots,f_n\rangle $ in which each generator $f_i $ is of bounded type, then $ I(S)\subset J(S) $ and $ J(S) = \overline{I(S)} $. 
\end{theorem}

In contrary of the result of the theorem \ref{1d}, the closure of escaping set $ I(S) $, that is,  $ \overline{I(S)} $ and the Julia set $ J(S) $ are S-completely invariant even though the semigroup $ S $ is not abelian. Dinesh Kumar and Sanjay kumar {\cite[Theorem 2.4 and Corollary 2.5]{kum1}} proved the following result.

\begin{theorem}\label{ci11}
The closure of the escaping set $ I(S) $ and Julia set $ J(S) $ are S-completely invariant if a semigroup $ S $ is finitely generated and each of its generators is of bounded type. 
\end{theorem}

\section{Completely invariant escaping set}

The theorem \ref {1d} is a condition for completely invariant escaping set of transcendental semigroup.  It is indeed a generalization of completely invariant property of classical escaping set of single function to more general settings of semigroup.  In this section, we generalize the classical completely  invariant notion of escaping set of single function to the completely  invariant  notion of escaping set of transcendental semigroup. That is, we make completely invariant escaping set of  semigroup $ S $  in different manner under each element of $ S $ and we show that such type of completely invariant escaping set is same as escaping set $ I(S) $ if and only if $ I(S) $ is S-completely invariant. 
\begin{dfn}\label{eq11}
For a transcendental semigroup $ S $, let us define the completely invariant escaping set of $ S $ 
\begin{align*}
K(S) =\bigcap_{i \in \mathbb{N}}\{E_{i}: E_{i}\; \text{is completely invariant  under each}\;  f \in S \; \text{and each}\; E_{i} \; \\  \text{contains points }  z \in \mathbb{C} \; \text{such that} \; f^{n}(z) \to \infty \;\text{as} \; n \to \infty  \; \text{for every}\; f \in S\}
\end{align*}
\end{dfn}
There are non-trivial transcendental semigroups $ S $ for which completely invariant escaping set  $K(S) $ exists. The following example stated in the theorem \ref{1c1} will be a good source of several other examples.
\begin{exm}\label{ex1}
Suppose that $ S = \langle f, g \rangle $ and $ I(f) = I(g) $. Then $ K = I(S) $.
\end{exm} 
Since $ I(f) $ is completely invariant under $ f $ and $ I(g) $ is completely invariant under $ g $. If $ I(f) = I(g) $, then $ I(h) = I(f) = I(g) = I(S) $ for all $ h \in S $. In this case, $ K(S) = I(S) $. A concrete nice example of this case is a semi group $ S = \langle f, g \rangle $ generated by the functions $ f(z) = e^{\lambda z}, \; \lambda \in \mathbb{C} \setminus \{0\}$ and $ g(z) = f^{k} + \frac{2\pi i}{\lambda}, \; k \in \mathbb{N} $. In this, we can find that $ I(h) = I(f) = I(g) = I(S) $ for all $ h \in S $. Note that the semigroup $ S $ that we mentioned here is not abelian. Another example of same kind is a semigroup $ S = \langle f, g \rangle $ generated by the functions $ f(z) = \lambda \sin z, \; \lambda \in \mathbb{C} \setminus \{0\}$ and $ g(z) = f^{k} + 2 \pi, \; k\in \mathbb{N} $. In this case, we also get $ K(S) = I(S) $.

From this example \ref{ex1}, we can conclude that escaping set $ I(S) $ may be completely invariant even if semigroup $ S $ is not abelian. In such a case, escaping set $ I(S) $ is nothing other than the set $ K(S) $. There are tanscendental semigroups $ S $ for which the set $ I(S) $ as well as the set $ K(S) $ might be empty. For example:
\begin{exm}
Suppose that $ S = \langle f, g \rangle $, where $ f(z) = e^{z} $ and $ g(z) = e^{-z} $. Then both $ K(S)$ and $I(S)$ are empty sets. For $ z \in I(f) $, then $ g(f^{n}(z)) = 1/e^{f^{n}(z)} = 1/f(f^{n}(z)) = 1/f^{n+1}(z) \to 0 $ as $ n \to \infty $. 
\end{exm}

\begin{prop}\label{ci1}
Let $ S $ be a transcendental semigroup. If $ K(S) \neq \emptyset $,  then it is S-completely invariant under each $ f \in S $ and it contains all points $ z \in \mathbb{C} $ with  $f^{n}(z) \to \infty $ as $ n \to \infty $ for every $ f \in S $. 
\end{prop}
\begin{proof}
By the definition \ref{eq11}, the set $ K(S) = \bigcap_{i \in \mathbb{N}} E_{i} $,  where $ E_{i} $ is completely invariant for each $ f \in S $ and it contains points $z\in \mathbb{C}  $ such that $ f^{n}(z) \to \infty $ as $ n \to \infty $. For all $ f \in S $, we have $ f(K(S)) = f(\bigcap_{i \in \mathbb{N}} E_{i}) \subseteq \bigcap_{i \in \mathbb{N}}f(E_{i}) \subseteq \bigcap_{i \in \mathbb{N}}E_{i} \subset  E_{i}$ for all $ i \in \mathbb{N} $. This implies that $ f(K(S)) \subset\bigcap_{i \in \mathbb{N}} E_{i} = K(S) $. Backward invariant of $ K(S) $ follows similarly. This proves that $ K(S) $ is S-completely invariant. The next part follows from the definition \ref{eq11}.
\end{proof}

\begin{prop}\label{ci}
The set $ K(S) \subset I(f)$ for each $ f\in S $ and so $ K(S) \subset \bigcap_{f \in S}I(f) $
\end{prop}
\begin{proof}
Let $ z \in K(S) $. Then by the definition \ref{eq11},   $ z \in E_{i} $ for all $ i $ and $f^{n}(z) \to \infty $ as $ n \to \infty $ for every $ f \in S $. This proves that $ z \in I(S) $. From the theorem \ref{1c}, we have $ I(S) \subset I(f) $ for all $ f \in S $. Thus, we must have $ K(S) \subset I(f)$ for all $ f \in S $. The last inclusion is obvious from the assertion $ K(S) \subset I(f)$ for all $ f \in S $. 
\end{proof}

The fact of this proposition \ref{ci} together with fact $ I(S) \subset I(f) $ for all $ f \in S $, we can say that $ K(S) $ is completely invariant escaping set which is contained in $ I(S) $. The following result shows that  $ K(S) $ sometime may be equal to $ I(S) $.  
\begin{prop}
Let $ S $ be a non-trivial transcendental semigroup. Then $ K(S)  \\ = I(S) $ if and only if $ I(S) $ is S-completely invariant.
\end{prop}
\begin{proof}
Let $ K(S) =I(S) $ where $ K(S) $ is S-completely invariant by the proposition \ref{ci1}. So, $ I(S) $ is S-completely invariant.

Conversely suppose that $ I(S) $ is S-completely invariant. Then any $ z \in I(S) $ implies that $ z \in K(S) $. So,  $ I(S) \subset K(S) $. On the other hand,  $ K(S) $ is completely invariant and consists of points $ z \in \mathbb{C} $ such that $ f^{n}(z) \to \infty $ as $ n \to \infty $ for every $ f \in S $.  So, it is nothing other than set $ I(S) $.
\end{proof}

A nice example of this proposition is an example \ref{ex1}. 
Next, we construct the following set which gives an alternative convenient description of the set $ K(S) $. 
Let $ S $ be a  transcendental semigroup for which the set $ K(S) \neq \emptyset $. Then by above  proposition \ref{ci}, we can write 
$$ 
K(S) \subset I(h)\; \text{for all}\;  h \in S 
$$ 
Note that as each $ I(h) $ is completely invariant, so their intersection $\bigcap_{h\in S} I(h) $ is also completely invariant. Define 
$$
E_{0}  = \bigcap_{h\in S} I(h)
$$
$$
E_{1} = \bigcup_{h \in S}h^{-1}(E_{0}) \cup \bigcup_{h \in S}h(E_{0}) 
$$
$$
\ldots \;\;\;\;\; \ldots \;\;\;\;\;  \ldots \;\;\;\;\; \ldots \;\;\;\;\; \ldots
$$
$$
E_{n +1} = \bigcup_{h \in S}h^{-1}(E_{n}) \cup \bigcup_{h \in S}h(E_{n}) 
$$
and 
\begin{equation}\label{eq12}
E = \bigcap_{n \in \mathbb{N}\cup\{0\}} E_{n} 
\end{equation}
Here, we have built-up set $ E $ from sets $ I(h)$ for all $ h \in S $ within their intersection.
\begin{prop}\label{cv0}
The set $E = \bigcap_{n \in \mathbb{N}\cup\{0\}} E_{n} $ is non-empty. 
\end{prop}
\begin{proof}
We show that $ K(S) \subset E_{n} $ for every $ n \in \mathbb{N} \cup \{0\} $ by induction. $ K(S) \subset E_{0} $ is obvious by above construction. By the completely invariant property of  $ K(S) $ under each $ h \in S, \; K(S) $ is subset of each sets $ h^{-1}(E_{0}) $ and $ h(E_{0}) $ for all $ \ h \in S$. This shows $ K(S) \subset E_{1} $. Let us suppose $ K(S) \subset E_{n} $. Since $ E_{n +1} = h^{-1}(E_{n}) \cup h(E_{n})  $ for all $ h \in S $. By the similar fashion as above, $ K(S) $ is subset of each of the sets $ h^{-1}(E_{n}) $ and $ h(E_{n}) $ for all $ \ h \in S $. This shows that $ K(S) \subset E_{n+1} $ for each  $ n \in \mathbb{N} \cup \{0\} $. This proves set $ E \neq \emptyset $. 
\end{proof}
 
The following result  will be a convenient description of a completely invariant escaping set of transcendental semigroup 

\begin{theorem}\label{cv1}
Let $ S $ be a transcendental semigroup. Then $ E = K(S) $, where $ E $ is a set defined in \ref{eq12}. 
\end{theorem}
First we prove the following lemma.
\begin{lem}\label{cv2}
 The  closure $ \overline{E} $ of any $ E\subset \mathbb{C} $ is completely invariant under a transcendental entire function   $ f $ if and only if the set $ E $ itself is completely invariant under the same function $ f $.
\end{lem}
\begin{proof}
Let $ \overline{E} $ is completely invariant under the given transcendental entire function $ f $. Then $ f (\overline{E}) \subset  \overline{E} $  and $ f^{-1} (\overline{E}) \subset  \overline{E} $. Let $ z \in \overline{E} $, then $ f(z) \in f(\overline{E}) $  and so
$ f(z) \in \overline{E}  $. Also, $ z \in \overline{E} \Longrightarrow $ there exists sequence $( z_{n})_{n \in \mathbb{N}} $ in $ E $ such that $ z_{n}\to z $ as $ n \to \infty $. From the continuity of the function $ f $ we can write $ f(z_{n}) \to f(z) $ as $ n \to \infty $. As  $ f(z) \in \overline{E} $,  we must have $ f(z_{n}) \in E $. Note that $ f(z_{n}) \in f(E) $ as $ z_{n} \in E $. Thus we must have $ f(E) \subset E $. 

Next, let $ z \in \overline{E} $, then $ f^{-1}(z) \in f^{-1}(\overline{E}) \subset \overline{E}   $. So there exists $ f^{-1}(z_{n}) \in E $ such that $  f^{-1}(z_{n})\to f^{-1}(z) $ as $ n \to \infty $. However, it is obvious that $f^{-1}(z_{n}) \in f^{-1}(E)  $. Thus we must have $ f^{-1}(E) \subset E $.

The converse part of this lemma follows from {\cite[Theorem 3.2.3]{bea}}.
\end{proof}

Note that under the assumption of this lemma \ref{cv2},  not only the closure of completely invariant set is invariant but also its complement, interior and boundary are also completely invariant (see for instance {\cite[Theorem 3.2.3]{bea}}).

\begin{proof}[Proof of the Theorem \ref{cv1}]
Since as defined in \ref{eq11}, $ K(S) $ is completely invariant under each of $ h\in S $ and is contained in $ I(h) $ for all $ h \in S $. Hence, it is contained in $ E_{n} $ for all $ n \in \mathbb{N} \cup \{0\} $. Therefore, $ K(S) \subset E$.
 
On the other hand, set $ E $ is contained in $ I(h) $ where each of $ I(h) $ is completely invariant. We need to show that $ E $ is completely invariant for each $ f \in S $ and for every $ z \in E,\;\; f^{n}(z) \to \infty $ as $ n \to \infty $ for every $ f \in S $. 

Since any $ f \in S $ is continuous in $ \mathbb{C} $ and $ E\subset \mathbb{C} $. So by the usual topological argument, $ f(\overline{E}) \subset \overline{f(E)} \Rightarrow f^{-1}(\overline{f(E)}) $ is closed in $ \mathbb{C} $ for every $ f \in S \Rightarrow f $ is a continuous closed map. This shows that $ f(\overline{E}) $ and $ f^{-1}(\overline{E}) $ are both closed sets  in $ \mathbb{C} $.  Since each $  f \in S $ is continuous closed map  and $ f(\overline{E_{n}})\subset \overline{E_{n}} $ and $ f^{-1}(\overline{E_{n}})\subset \overline{E_{n}} $ for all $ n$. Which shows that $ f(\overline{E}) \subset \overline{E} $ and $ f^{-1}\overline{(E)} \subset \overline{E} $. By lemma \ref{cv2}, it proves that $ E $ is completely invariant under each $ f \in S $.

Finally,  any $ z \in E \Rightarrow z \in E_{n} $ for all $ n $. Again $ E_{n} $ is a union of all images and pre-images of $ E_{n -1} $ under the  each map $ f \in S$. By this way, the point $ z $ belongs to the image or pre-image of $ E_{0} $ under each map $ f\in S$. Since $ E_{0} $ is contained within  of all escaping set of all functions $ f \in S $, so this point $ z $ goes to infinity under the iteration of any function in $ S $. Hence $ E \subset K(S) $.
\end{proof}

The following relation holds good in general between the sets $ K(S) $ and $ I(S) $ for a non-abelian transcendental semigroup $ S $.
\begin{prop}
Let $ S $ be a non-trivial and non-abelian transcendental  semigroup. Then $ K(S) \subset I(S) $.
\end{prop}
\begin{proof}
By the theorem \ref{cv1}, $ K(S) = E $ and $ E \subset I(S) $ by the construction of set $ E $ (that is, $E$ is completely invariant and $ I(S) $ is not)  in \ref{eq12}. Hence the proposition follows.
\end{proof}

We also can construct $ I(S) $ by similar fashion as in $ K(S) $. Note that escaping set $ I(S) $ is in general only forward invariant. Define 
 $$
F_{0}  = \bigcap_{h\in S} I(h), 
$$
$$
F_{1} =  \bigcup_{h \in S}(h(F_{0})) 
$$
$$
\ldots \;\;\;\;\;\; \ldots \;\;\;\;\;\; \ldots
$$
$$
F_{n +1} =  \bigcup_{h \in S}(h(F_{n})) 
$$
and 
\begin{equation}\label{eq13}
F = \bigcap_{n \in \mathbb{N}} F_{n} 
\end{equation}
We can show that $ F \neq \emptyset $ by the similar process as in proposition \ref{cv0}.
Here, we also have built-up set $ F $ from sets $ I(h)$ for all $ h \in S $ within their intersection. The fundamental difference of this set $ F $ to that from the set $ E $ of \ref{eq12} is that it is constructed from only forward invariant property under each $ h \in S $. Where as the set $ E $ was constructed by completely invariant property under each  $ h \in S $. The following theorem  provides an alternative definition of escaping set of transcendental semigroup.
\begin{theorem}
Let $ S $ be a transcendental semigroup. Then $ F = I(S) $, where $ F $ is defined in \ref{eq13}.
\end{theorem}  
\begin{proof}
Since $ I(S) $ is forward invariant under each $ h\in S $ and is contained in each of $ I(h) $, so it is contained in $ F $. On the other hand, the set $ F $ is contained in $ I(S) $ by above construction.  
\end{proof}

\end{document}